\documentclass{article}
\usepackage{amsmath,amssymb}
\usepackage{amsfonts}
\usepackage{amsthm}
\usepackage{bm}
\usepackage{color}
\usepackage{graphicx}
\usepackage{hyperref}

\newcommand{\ra}[1]{\renewcommand{\arraystretch}{#1}}

\newtheorem{theorem}{Theorem}

\newtheorem{corollary}[theorem]{Corollary}

\newtheorem{example}[theorem]{Example}

\newtheorem{lemma}[theorem]{Lemma}

\newtheorem{proposition}[theorem]{Proposition}
\newtheorem{remark}[theorem]{Remark}

\begin{document}

\title{A Quantitative Functional Central Limit Theorem for Shallow Neural
Networks}
\author{Valentina Cammarota\thanks{%
Department of Statistics, University of Rome La Sapienza, email:
valentina.cammarota@uniroma1.it}, Domenico Marinucci\thanks{%
Department of Mathematics, University of Rome Tor Vergata, email:
marinucc@mat.uniroma2.it}, \and Michele Salvi\thanks{%
Department of Mathematics, University of Rome Tor Vergata, email:
salvi@mat.uniroma2.it} and Stefano Vigogna\thanks{%
Department of Mathematics, University of Rome Tor Vergata, email
vigogna@mat.uniroma2.it}}
\maketitle

\begin{abstract}
We prove a quantitative functional central limit theorem for one-hidden-layer neural networks with generic activation function. Our rates of
convergence depend heavily on the smoothness of the
activation function, and they range from logarithmic for non-differentiable
non-linearities such as the ReLu to $\sqrt{n}$ for highly regular activations. Our main
tools are based on functional versions of the Stein-Malliavin method; in
particular, we rely on a quantitative functional central limit
theorem which has been recently established by Bourguin and Campese (2020).

\begin{itemize}
\item Keywords and Phrases: Quantitative Functional Central Limit Theorems,
Wiener-Chaos Expansions, Neural Networks, Gaussian Processes

\item AMS Classification: 60F17, 68T07, 60G60
\end{itemize}
\end{abstract}

\section{Introduction and background \label{Introduction}}

In this paper, we shall be concerned with one-hidden-layer neural networks with
Gaussian initialization, that is random fields $F:\mathbb{S}%
^{d-1}\rightarrow \mathbb{R}$ of the form
\begin{equation*}
F(x)=\frac{1}{\sqrt{n}}\sum_{j=1}^{n}V_{j}\sigma \left(\sum_{\ell =1}^{d}W_{j\ell
}x_{\ell }\right)=\frac{1}{\sqrt{n}}\sum_{j=1}^{n}V_{j}\sigma (W_{j}x)\text{ , }
\end{equation*}%
where $V_{j} \in \mathbb{R} ,W_{j} \in \mathbb{R}^{1\times d}$ are, respectively, random variables and vectors, whose entries are independent Gaussian with
zero mean and variance $\mathbb{E}[ V_{j}^{2}] =\mathbb{E}[
W_{j\ell }^{2}] =1,$ $j=1,...,n,$ $\ell =1,...,d.$ Here, $\sigma:
\mathbb{R\rightarrow }\mathbb{R}$ is an activation function whose properties
and form we will discuss below.

The random field $F$ is defined on the unit sphere $\mathbb{S}^{d-1},$ with
zero mean and covariance function such that, for any pair $x_{1},x_{2}\in
\mathbb{S}^{d-1}$,
\begin{equation*}
\left\{ S(x_{1},x_{2})\right\} :=\mathbb{E}[F(x_{1})F(x_{2})]=\mathbb{E}%
[\sigma (W_{j}x_{1})\sigma (W_{j}x_{2})]\text{ .}
\end{equation*}%
The function $\left\{ S(\cdot,\cdot)\right\} $ is actually isotropic and we can write%
\begin{equation*}
\left\{ S(x_{1},x_{2})\right\} =:\Gamma (\left\langle
x_{1},x_{2}\right\rangle )\qquad \text{ for some function }\Gamma
:[-1,1]\rightarrow \mathbb{R}\text{ ,}
\end{equation*}%
where $\left\langle x_{1},x_{2}\right\rangle $ denotes the standard Euclidean
inner product. For the sake of brevity and simplicity, in this paper we
restrict our attention to univariate neural networks; the extension to the
multivariate case can be obtained along similar lines, up to a normalizing constant depending
on the output dimension.

Our aim in this paper is to establish a quantitative functional central
limit theorem, that is, to study the distance with a suitable probability
metric between the random field $F$ and a Gaussian random field on the
sphere $\mathbb{S}^{d-1}$ with mean zero and covariance function $S:%
\mathbb{S}^{d-1}\times \mathbb{S}^{d-1}\rightarrow \mathbb{R}$.

The distribution of neural networks under random initialization is a
classical topic in learning theory, the first result going back to the
groundbreaking work \cite{Neal}, where it was proved that a central limit theorem
holds as the width of the network diverges to infinity. Much more recently,
a few authors have investigated the speed of convergence to the Gaussian
limit distribution. In this respect, some influential papers \cite{Hanin2021,RobertsHanin,Yaida} have studied the behaviour
of higher-order cumulants, covering also deep neural networks. Quantitative
central limit theorems in suitable probability metrics have been considered
very recently in \cite{BasteriTrevisan} and \cite{Bordino}. The former
authors have proved a finite-dimensional quantitative central limit theorem
for neural networks of finite depth whose activation functions satisfy a
Lipschitz condition; in \cite{Bordino}, the authors have proved second-order Poincar\'{e}
inequalities (which imply one-dimensional quantitative central limit
theorems) for neural networks with $C^{2}$ activation functions.

Understanding the Gaussian behaviour of a neural network allows, for
instance, to investigate the geometry of its landscape, e.g. the cardinality of
its minima, the number of nodal components and many other quantities of interest. However,
convergence of the finite-dimensional distributions is in general not
sufficient to constraint such landscapes. For this reason, functional
results, that is, bounds on the speed of convergence in functional spaces,
are also of great interest. So far, the literature on quantitative
functional central limit theorems is still limited: \cite{Eldan} and \cite%
{Klukowski} have focused on one-hidden-layer networks, where the random
coefficients in the inner layer are Gaussian for \cite{Eldan} and uniform on
the sphere for \cite{Klukowski}, whereas the coefficients in the outer layer
follow a Rademacher distribution for both. In particular, the authors in
\cite{Eldan} manage to establish rates of convergence in Wasserstein
distance which are (power of) logarithmic for ReLu and other activation
functions, and algebraic for polynomial or very smooth activations, see below
for more details. On the other hand, the rates in \cite{Klukowski} for ReLu
networks are of the form $O(n^{-\frac{1}{2d-1}})$; this is algebraic for
fixed values of $d,$ but it can actually converge to zero more slowly than the inverse of a logarithm if $d$
is of the same order as $n,$ as it is the case for many
applications.

\subsection{Purpose and plan of the paper}

We consider in this work functional quantitative central limit theorems
under general activations and for coefficients that are Gaussian for both
layers, which seems the most relevant case for applications; our approach is
largely based upon very recent results by \cite{BourguinCampese} on
Stein-Malliavin techniques for random elements taking values in Hilbert
spaces (we refer to \cite{NourdinPeccatiPTRF,NourdinPeccati} for the
general foundations of this approach, together with \cite{LedouxPeccati,BourguinCampeseLeonenkoTaqqu,Azmoodeh,DoblerPeccati}
for some more recent references). Our main results are collected in Section %
\ref{SectionMainTheorem}, whereas their proof with a few technical lemmas
are given in Section \ref{SectionProofMain}. A short comparison with the
existing literature is provided in Section \ref{SectionComparison}. The
Appendix \ref{Appendix} is mainly devoted to background results which we heavily exploit
throughout the paper.
\paragraph{Notation.}
Hereafter, we will write $%
a_{n}\sim b_{n}$ for two positive sequences such that $\lim_{n\rightarrow
\infty }a_{n}/b_{n}=1$.
The expression $ A \lesssim B $ means that $ A \leq C B $ for some absolute constant $C>0$.
We will denote by $ \|\cdot\| $ the $L^2$ norm corresponding to the uniform probability measure on the unit sphere $ \mathbb{S}^{d-1} $.

\subsection{Acknowledgements}

This paper has originated by a very inspiring short course taught by Boris Hanin at the University of Rome Tor Vergata in January 2023; we are very grateful to him for many deep insights and illuminating conversations. We also acknowledge financial support from MUR
Department of Excellence Programme \emph{MatModTov}, MUR Prin project \emph{Grafia}, Indam/GNAMPA and PNRR CN1 High Performance Computing, Spoke 3.

\section{Main results \label{SectionMainTheorem}}

In order to state our main theorems, we shall need some further
assumptions and notations. We shall always be concerned with activation
functions which are square integrable with respect to the standard Gaussian
measure, i.e., such that
\begin{equation*}
\mathbb{E}[\sigma ^{2}(Z)]<\infty \text{ , }Z\sim N(0,1)\text{ ;}
\end{equation*}%
this is a truly minimal conditions, which is guaranteed by $\sigma
(z)=O(\exp (z^{2}/(2+\delta ))$ for all $\delta >0.$ For such activation
functions, it is well-known that the following Hermite expansion holds, in
the $L^{2}$ sense with respect to Gaussian measure (see e.g. \cite%
{NourdinPeccati}):%
\begin{equation*}
\sigma (x)=\sum_{q=0}^{\infty }J_{q}(\sigma )\frac{H_{q}(x)}{\sqrt{q!}}\text{
, where }H_{q}(x):=(-1)^{q} e^\frac{x^{2}}{2}\frac{d^{q}}{dx^{q}} e^{-%
\frac{x^{2}}{2}}\text{ ,}
\end{equation*}%
where $\left\{ H_{q}\right\} _{q=0,1,2,...,}$is the well-known sequence
of Hermite polynomials. The coefficients $J_{q}(\sigma ),$ which will play a
crucial role in our arguments below, are defined according to the following
(normalized) projection:%
\begin{equation*}
J_{q}(\sigma ):=\frac{1}{\sqrt{q!}}\mathbb{E}[\sigma (Z)H_{q}(Z)]\text{ .}
\end{equation*}%
In the following, when no confusion is possible, we may drop the dependence of $%
J$ on $\sigma $ for ease of notation. We remark that our notation is
to some extent non-standard, insofar we have introduced the factor $\frac{1}{%
\sqrt{q!}}$ inside the projection coefficient $\mathbb{E}[\sigma
(Z)H_{q}(Z)];$ equivalently, we are defining the projection coefficients in
terms of Hermite polynomials which have been normalized to have unit
variance. Indeed, as well-known,%
\begin{equation*}
\mathbb{E}\left[ \left( \frac{H_{q}(Z)}{\sqrt{q!}}\right) ^{2}\right] =\frac{%
1}{q!}\mathbb{E}\left[ (H_{q}(Z))^{2}\right] =1\text{ .}
\end{equation*}

In short, our main results state that a quantitative functional central
limit theorem for neural networks built on $\sigma$ holds, and the rate
of convergence depends on the rate of decay of $\left\{ J_{q}(\sigma
)\right\} ,$ as $q\rightarrow \infty ;$ roughly put, it is logarithmic when
this rate is polynomial (e.g., the ReLu case), whereas convergence occurs at
algebraic rates for some activation functions which are smoother, with
exponential decay of the coefficients. A more detailed discussion of these
results and comparisons with the existing literature are given below in
Section \ref{SectionComparison}.

Let us discuss an important point about normalization. \emph{In this paper,
the measure on the sphere }$\mathbb{S}^{d-1}$\emph{\ is normalized to have
unit volume.} The bound we obtain are not invariant to this normalization,
and indeed they would be much tighter if the measure on the sphere was taken
as usual to be $s_{d}=\frac{2\pi ^{d/2}}{\Gamma (\frac{d}{2})},$ the surface
volume of $\mathbb{S}^{d-1}$. Indeed, by Stirling's formula%
\begin{equation*}
s_{d}=\frac{2\pi ^{d/2}}{\Gamma (\frac{d}{2})}\sim \frac{2\pi
^{d/2}2^{d/2}e^{d/2}}{\sqrt{\pi d}d^{d/2}}=\frac{2}{\sqrt{\pi d}}\left(
\sqrt{\frac{2e\pi }{d}}\right) ^{d}\text{ ;}
\end{equation*}%
$s_{d}$ achieves its maximum for $d=7$ ($s_{7}=33.073$) and decays faster
than exponentially as $d\rightarrow \infty $. This means that, without the
normalization that we chose, our bound on the $d_{2}$ metric would be
actually smaller by a factor of roughly $d^{-d/2}$ when the dimension grows.
On the other hand, if we were to take standard Lebesgue measure $\lambda $
then we would obtain, by a standard application of Hermite expansions and
the Diagram Formula%
\begin{equation*}
\mathbb{E}\left\Vert F\right\Vert _{L^{2}(\lambda
)}^{2}=\sum_{q}J_{q}^{2}\int_{\mathbb{S}^{d-1}}\lambda
(dx)=\sum_{q}J_{q}^{2}s_{d}\text{ ,}
\end{equation*}%
so that the $L^{2}$ norm would decay very quickly as $d$ increases, making
the interpretation of results less transparent.

Following \cite{BourguinCampese}, the convergence in our central limit theorem is measured in the $d_{2}$
metric. This is given by%
\begin{equation*}
d_{2}(F,Z)=\sup_{\left\Vert h\right\Vert _{C_{b}^{2}(L^{2}(\mathbb{S}%
^{d-1}))}\leq 1}\left\vert \mathbb{E}h(F)-\mathbb{E}h(Z)\right\vert \text{ ,
}
\end{equation*}%
where $C_{b}^{2}(L^{2}(\mathbb{S}^{d-1}))$ is the space of real-valued
applications on $L^{2}(\mathbb{S}^{d-1})$ (with respect to the uniform measure) with two bounded Frechet derivatives. It is to be
noted that the $d_{2}$ metric is bounded by the Wasserstein distance of
order 2, i.e.%
\begin{equation*}
d_{2}(F,Z)\leq \mathcal{W}_2(F,Z):=\inf_{(\widetilde{F},\widetilde{Z})}\left( \mathbb{%
E}\left\Vert \widetilde{F}-\widetilde{Z}\right\Vert _{L^{2}(\mathbb{S}%
^{d-1})}^{2}\right) ^{1/2}\text{ ,}
\end{equation*}%
where the infimum is taken over all the possible couplings of $(F,Z)$.

Our first main statement is as follows.

\begin{theorem}
\label{MainTheorem} Under the previous assumptions and notations, we have
that, for all $ Q \leq \log_3 \sqrt{n} $,%
\begin{equation}
d_{2}(F,Z)\leq C \| \sigma \| \frac{1}{\sqrt[4]{n}} \sqrt{\sum_{q=0}^{Q}J_{q}^{2}(%
\sigma )q 3^q} + \frac{3}{2} \sqrt{\sum_{q=Q+1}^{\infty }J_{q}^{2}(\sigma
)} \text{ ,}  \label{Mainformula}
\end{equation}
where $C$ is an absolute constant (in particular, independent of the
input dimension $d$),
and $ \|\sigma\| $ is the $L^2$ norm of $\sigma$ taken with respect to the Gaussian density on $\mathbb{R}$.
\end{theorem}

The proof is postponed to Section \ref{sec:proofthm1}.
From Theorem \ref%
{MainTheorem}, optimizing over the choice of $Q$, it is immediate to obtain much more explicit bounds.
In the case of polynomial decay of the Hermite coefficients,
the choice $ Q = \log n / (3\log 3) $ yields the following result.
\begin{corollary}
\label{MainCorollary} In the same setting as in Theorem \ref{MainTheorem},
for $ J_{q}(\sigma ) \lesssim q^{-\alpha} $, $ \alpha >\frac{1}{2} $,
we have%
\begin{equation*}
d_{2}(F,Z)\leq C \|\sigma\| \frac{1}{(\log n)^{\alpha -\frac{1}{2}}} .
\end{equation*}
\end{corollary}

\begin{example}[ReLu]
As shown in Lemma \ref{lem:relu}, for the ReLu activation $\sigma (t)=t%
\mathbb{I}_{[0,\infty )}(t)$ we have that $J_{q}(\sigma )\lesssim
q^{-\frac{5}{4}},$ whence we obtain the bound $d_{2}(F,Z)\lesssim (\log
n)^{-\frac{3}{2}}.$ Once again, we stress that the constant is independent
of the input dimension $d$.
\end{example}

The statement of Theorem \ref{MainTheorem} is given in order to cover the
most general activation functions, allowing for possibly
non-differentiable choices such as the ReLu. Under stronger conditions, the
result can be improved; in particular, assuming the activation function has a
Malliavin derivative with bounded fourth moment (i.e., it belongs to the
class $\mathbb{D}^{1,4},$ see \cite{NourdinPeccati,BourguinCampese}),
we obtain the following extension.

\begin{theorem}
\label{MainTheorem2} Under the previous assumptions and notations, and
assuming furthermore that $\sigma (Wx)\in \mathbb{D}^{1,4}$, we have that,
for all $Q \in \mathbb{N}$,%
\begin{equation}
d_{2}(F,Z)\leq C \frac{1}{\sqrt{n}} \sum_{q=0}^Q J_q^2(\sigma) q 3^q \left( \| \sigma \|^2 + \frac{1}{\sqrt{n}} \sum_{q=0}^Q J_q^2(\sigma) 3^q \right)
 + \frac{3}{2} \sqrt{\sum_{q=Q+1}^{\infty }J_{q}^{2}(\sigma )}\text{
,}  \label{MainFormula2}
\end{equation}
where $C$ is an absolute constant (in particular, independend of the
input dimension $d$),
and $ \|\sigma\| $ is the $L^2$ norm of $\sigma$ taken with respect to the Gaussian density on $\mathbb{R}$.
\end{theorem}

We prove Theorem \ref{MainTheorem2} in Section \ref{sec:proofthm2}. Again, imposing specific decay profiles on the Hermite expansion we can obtain explicit bounds.
In particular, when $ J_q \lesssim e^{-\beta q} $ with $ \beta > \log \sqrt{3} $,
the second sum appearing in \eqref{MainFormula2} stays finite for all $Q$,
hence the bound assumes the form
$$
d_{2}(F,Z)\leq C \| \sigma \|^2 \frac{1}{\sqrt{n}} \sum_{q=0}^Q J_q^2(\sigma) q 3^q
 + \frac{3}{2} \sqrt{\sum_{q=Q+1}^{\infty }J_{q}^{2}(\sigma )} \ ,
$$
more in line with the bound \eqref{Mainformula}.
In such a case, letting $Q$ go to infinity leads to the next result.

\begin{corollary}
\label{MainCorollary2} In the same setting as in Theorem \ref{MainTheorem2},
for $ J_{q}(\sigma )\lesssim e^{-\beta q}$, $\beta >\log \sqrt{3}$,
we have%
\begin{equation*}
d_{2}(F,Z)\leq C \frac{1}{\sqrt{n}} \ .
\end{equation*}
\end{corollary}

\begin{example}[polynomials/erf]
The assumptions of Corollary \ref{MainCorollary2} are fulfilled by polynomial
activations and by the error function $ \operatorname{erf}(t) = \tfrac{2}{\sqrt{\pi}} \int_0^t e^{-s^2} ds $, for which $ J_q^2(\sigma) \lesssim (2/3)^{q} $ -- cfr.~\cite{Klukowski}. In these
cases, the fact that $\sigma (Wx)\in \mathbb{D}^{1,4}$ can be readily shown
by means of the triangle inequality and the standard hypercontractivity
bound for Wiener chaos components -- see \cite[Corollary
2.8.14]{NourdinPeccati}.
\end{example}

\begin{example}[tanh/logistic]
Of course, other forms of decay could be considered. For instance, for the
hyperbolic tangent $\sigma (t)=(e^{t}-e^{-t})/(e^{t}+e^{-t})$ the rate of
decay of the Hermite coefficients is of order $\exp (-C\sqrt{q})$ (see e.g.
\cite{Eldan}), hence the result of Corollary \ref{MainCorollary2} does not
apply; the bounds in Corollary \ref{MainCorollary} obviously hold, but
applying directly Theorem \ref{MainTheorem} and some algebra we obtain the finer bound%
\begin{equation*}
d_{2}(F,Z)\lesssim  \exp(-c\sqrt{\log n})%
\qquad \text{for }J_{q}(\sigma )\leq \exp (-C\sqrt{q})\text{ }.
\end{equation*}%
The same bound holds also for the sigmoid/logistic activation function $%
\sigma (t)=(1+e^{-t})^{-1}$.
\end{example}

\subsection{Discussion} \label{sec:discussion}

The ideas in our proof are quite standard in the literature on Quantitative
Central Limit Theorems, and can be extended directly to the functional
setting that we consider here. As a first step, we partition our neural
network into two processes, one corresponding to its projection onto the
first $Q$ Wiener chaoses, for $Q$ an integer to be chosen below, and the
other corresponding to the remainder. This remainder can be easily bounded
in Wasserstein distance by standard $L^{2}$ arguments; for the leading term,
following recent results by \cite{BourguinCampese}, we need a careful
analysis of fourth moments and covariances for the $L^{2}$ norms of the
Wiener projections. In particular, these bounds can be expressed in terms of
multiple integrals of fourth-order cumulants, as by now standard in the
literature on the so-called Stein-Malliavin method (see \cite{NourdinPeccati,BourguinCampese} and the references therein). The computation of these
terms is technical, but the results are rather explicit; they are collected in
dedicated propositions and lemmas.

One technical point that we shall address is the following. The convergence
results by \cite{BourguinCampese} require the limiting process to be
nondegenerate; this condition is not always satisfied for arbitrary
activation functions if one takes the corresponding Hilbert space to be $%
L^{2}(\mathbb{S}^{d-1})$ (counter-examples being finite-order polynomials).
However, we note that for activations for which the corresponding networks
are dense in the space of continuous functions (such as the ReLu or the
sigmoid and basically all non-polynomials, see for instance the classical
universal approximation theorems in \cite{Cybenko,Hornik,Hornik2,Leshno,Pinkus}), then the nondegeneracy condition
is automatically satisfied. On the other hand, when the condition fails our
results continue to hold, but the underlying functional space must be taken
to be the reproducing kernel Hilbert space generated by the covariance
operator, which is strictly included into $L^{2}(\mathbb{S}^{d-1})$ when
universal approximation fails (e.g., in the polynomial case).

\section{A comparison with the existing literature \label{SectionComparison}}

Two papers that have established quantitative functional central limit theorems for neural networks
are those by \cite{Eldan} and \cite{Klukowski}. Their settings and results
are not entirely comparable to ours; on the one hand, they use the Wasserstein
distance, which is slightly stronger that the $d_{2}$ metric we
consider here.
On the other hand, their model for the random weight is different from ours:
for the outer layers, both consider Rademacher variables,
while for the inner layer the distribution is Gaussian in \cite{Eldan}
and uniform on the sphere in \cite{Klukowski};
on the contrary, we assume Gaussian distribution for both inner
and outer layer.
As a further (minor) difference, we note that in \cite%
{Eldan}, as well as in our paper, input variables are in $\mathbb{S}^{d-1}$,
while
\cite{Klukowski} considers $\sqrt{d}\mathbb{S}^{d-1}$;
this is just a notational issue, though, because in \cite%
{Klukowski} the argument of the activation function is normalized by a
factor $1/\sqrt{d}$.

Even with these important caveats, it is nevertheless of some interest to
compare their bounds  with ours, for activation functions for which
there is an overlap. We report their results together with ours in Table %
\ref{savare} (the constant $C$ may differ from one box to the other, but in
all cases it does not depend neither on $d$ nor on $n$).

\begin{table}[h]
\centering
\ra{2}
\resizebox{\columnwidth}{!}{%
\begin{tabular}{llll}
 & Eldan et al. \cite{Eldan} & Klukowski \cite{Klukowski} & This paper \\ \hline
$J_{q}\sim q^{-\alpha }$ & $(\frac{\log n}{\log \log n\log d})^{-\alpha +%
\frac{1}{2}}$ & - & $(\log n)^{-\alpha +\frac{1}{2}}$\\ \hline
ReLu & $(\frac{\log n}{\log \log n\log d})^{-\frac{3}{4}}$ & $n^{-\frac{3}{%
4d-2}}$ & $(\log n)^{-\frac{3}{4}}$ \\ \hline
$\tanh$ / logistic & $\exp(-c\sqrt{\frac{\log n}{\log d\log \log n}})$ & -
& $\exp(-c\sqrt{\log n})$ \\ \hline
$ J_{q} \sim e^{-\beta q} $ & $ n^{-c (\log\log n \log d)^{-1}} $ & - & $n^{-\frac{1}{2}}$ \\ \hline
$\operatorname{erf}$ & $ n^{-c (\log\log n \log d)^{-1}} $ & $ C^d (\log n)^{\frac{d}{2}-1} n^{-\frac{1}{2}} $  & $n^{-\frac{1}{2}}$ \\ \hline
polynomial order $p$ & $p^{c p}d^{\frac{5p}{6}-\frac{1}{12}}n^{-\frac{1}{6}}$ & $%
(d+p)^{\frac{d}{2}}  n^{-\frac{1}{2}}$ & $n^{-\frac{1}{2}}$ \\ \hline
\end{tabular}%
}
\caption{Comparison of convergence rates established by different functional quantitative central limit theorems for several activation functions. Bear in mind that two different metrics $ d_2 \leq \mathcal{W}_2 $ are considered, $ \mathcal{W}_2 $ for \cite{Eldan,Klukowski}, and $d_2$ for this paper.
The parameters $\alpha$ and $\beta$ must satisfy $ \alpha > 1/2 $ and $ \beta > \log \sqrt{3} $.
}
\label{savare}
\end{table}

Comparing to \cite{Eldan}, our bounds remove
a logarithmic factor in the input dimension and a $\log \log $ factor in
the number of neurons for ReLu and $\tanh$ networks;
for smooth activations, the rate goes from $n^{-1/6}$ to $%
n^{-1/2}$,  and the constants lose the polynomial dependence on the dimension.
The rate in \cite{Klukowski} in the polynomial case is $ n^{-1/2} $ as ours,
but with a factor growing in the input dimension $d$ as $d^{d/2}$.
In the ReLu setting, \cite{Klukowski} displays the algebraic rate $n^{-\frac{3}{4d-2}}$,
which \emph{for fixed values of }$d$
decays faster than our logarithmic bound.
However,
interpreting these bounds from a ``fixed $d,$ growing $n$'' perspective can be
incomplete: when considering distances in
probability metrics it is of interest to allow both $d$ and $n$ to vary. In
particular, for neural networks applications, it is often the case
that the input dimension and number of neurons are of comparable order; taking for
instance $d=d_{n}\sim n^{\alpha }$, it is immediate to verify that for all $%
\alpha >0$ (no matter how small) one has%
\begin{equation*}
\lim_{n\rightarrow \infty }\frac{(\log n)^{-\frac{3}{4}}}{n^{-\frac{3}{4d-2}}%
}=\lim_{n\rightarrow \infty }\frac{(\log n)^{-\frac{3}{4}}}{\exp (-\frac{3}{%
4n^{\alpha }-2}\log n)}=0\text{ ,}
\end{equation*}%
so that our bound in the $d_{2}$ metric decays faster that the one by \cite%
{Klukowski} in $\mathcal{W}_2$ under these circumstances.

\section{Proof of the main results \label{SectionProofMain}}

Our main results, Theorems \ref{MainTheorem} and \ref{MainTheorem2}, are proved in Sections \ref{sec:proofthm1} and \ref{sec:proofthm2}, respectively.
The proofs use auxiliary propositions and lemmas,
which are established in Sections \ref{sec:M(F)} and \ref{sec:C(F)}.

\subsection{Proof of Theorem \protect\ref{MainTheorem}} \label{sec:proofthm1}

The main idea behind our proof is as follows. For some integer $Q$ to be
fixed later, write
\begin{equation*}
F=F_{\leq Q}+F_{>Q}\text{ , \ }
\end{equation*}%
where
$$
 F_{\leq Q}:=\sum_{q=0}^{Q}F_{q} \ , \qquad %
F_{>Q}:=\sum_{q=Q+1}^{\infty }F_{q}\text{ ,}
$$
and
\begin{equation*}
F_{q}(x):=\frac{J_{q}(\sigma )}{\sqrt{n}}\sum_{j=1}^{n}V_{j}\frac{%
H_{q}(W_{j}x)}{\sqrt{q!}}\text{ , }x\in \mathbb{S}^{d-1}.\text{ }
\end{equation*}%
In words, as anticipated in the Section \ref{sec:discussion}, we are partitioning our
network into a component projected onto the $Q$ lowest Wiener chaoses and
the remainder projection on the highest chaoses. Now let us denote by $Z$
a zero mean Gaussian process with covariance function%
\begin{equation*}
\mathbb{E}\left[ Z(x_{1})Z(x_{2})\right] :=\sum_{q=0}^{\infty
}J_{q}^{2}\left\langle x_{1},x_{2}\right\rangle ^{q}.
\end{equation*}%
Likewise, in the sequel we shall write $\left\{ Z_{q}\right\} _{q\in\mathbb{N}}$
for a sequence of independent zero mean Gaussian with covariance function $%
\mathbb{E}\left[ Z_{q}(x_{1})Z_{q}(x_{2})\right] := J_{q}^{2}\left\langle
x_{1},x_{2}\right\rangle ^{q}.$ \ Our idea is to use Theorem 3.10 in \cite%
{BourguinCampese} and hence to consider%
\begin{eqnarray*}
d_{2}(F,Z) &\leq &d_{2}(F_{\leq Q},Z)+d_{2}(F,F_{\leq Q}) \\
&\leq &\frac{1}{2}\left( \sqrt{M(F_{\leq Q})+C(F_{\leq Q})}+\left\Vert
S-S_{\leq Q}\right\Vert _{L^{2}(\Omega ,\rm{HS})}\right) +\mathcal{W}_2(F,F_{\leq Q})%
\text{ ,}
\end{eqnarray*}%
where%
\begin{align*}
M(F_{\leq Q}) & := \frac{1}{\sqrt{3}}\sum_{p,q}^{Q}c_{p,q}\sqrt{\mathbb{E}%
\left\Vert F_{p}\right\Vert ^{4}(\mathbb{E}\left\Vert F_{q}\right\Vert ^{4}-%
\mathbb{E}\left\Vert Z_{q}\right\Vert ^{4})} \text{ ,} \\
C(F_{\leq Q}) & := \sum_{\substack{ p,q  \\ p\neq q}}^{Q}c_{p,q}\operatorname{Cov}(\left\Vert
F_{p}\right\Vert ^{2},\left\Vert F_{q}\right\Vert ^{2})\text{ ,} \\
c_{p,q} & := \begin{cases}
1+\sqrt{3} & p=q \\
\frac{p+q}{2p} & p\neq q \ .
\end{cases}
\end{align*}

Now, we have
\begin{equation*}
\mathcal{W}_2(F,F_{\leq Q})\leq \sqrt{\sum_{q=Q+1}^{\infty }J_{q}^{2}} .
\end{equation*}%
Moreover,
\begin{equation*}
\left\Vert S-S_{\leq Q}\right\Vert _{L^{2}(\Omega ,\rm{HS})}^{2}\leq
\sum_{q=Q+1}^{\infty }J_{q}^{2} \ .
\end{equation*}%
Indeed, first note that the covariance operator can be written explicitly in coordinates as%
\begin{align*}
S_{\leq Q}(x_{1},x_{2}) &= \frac{1}{n}\sum_{p,q}^{Q}J_{p}(\sigma )J_{q}(\sigma
)\sum_{j_{1},j_{2}=1}^{n}\mathbb{E}\left[ \left\{
V_{j_{1}}H_{p}(W_{j_{1}}x_{1})\right\} \left\{
V_{j_{2}}H_{q}(W_{j_{2}}x_{2})\right\} \right] \\
&= \sum_{q}^{Q}J_{q}^{2}(\sigma )\left\langle x_{1},x_{2}\right\rangle ^{q} \ ,
\end{align*}%
and hence
\begin{equation*}
S(x,y)-S_{\leq Q}(x,y)=\sum_{q=Q+1}^{\infty }J_{q}^{2}\left\langle
x,y\right\rangle ^{q} \ .
\end{equation*}%
Therefore, taking the standard basis of spherical harmonics $\left\{ Y_{\ell
m}\right\}$, which are eigenfunctions of the covariance operators (see \cite%
{MaPeCUP}),%
\begin{eqnarray*}
&&\left\Vert S-S_{\leq Q}\right\Vert _{L^{2}(\Omega ,\rm{HS})}^{2} \\
&=&\sum_{\ell ,\ell ^{\prime },m,m^{\prime }}\sum_{q=Q+1}^{\infty }\int_{%
\mathbb{S}^{d-1}\times \mathbb{S}^{d-1}}Y_{\ell m}(x)Y_{\ell ^{\prime
}m^{\prime }}(y)\sum_{\ell ^{\prime \prime }m^{\prime \prime }}C_{\ell
^{\prime \prime }}(q)Y_{\ell ^{\prime \prime }m^{\prime \prime }}(x)Y_{\ell
^{\prime \prime }m^{\prime \prime }}(y)dxdy \\
&=&\sum_{\ell ,\ell ^{\prime },m,m^{\prime }}\sum_{q=Q+1}^{\infty
}\sum_{\ell ^{\prime \prime }m^{\prime \prime }}C_{\ell ^{\prime \prime
}}(q)\int_{\mathbb{S}^{d-1}\times \mathbb{S}^{d-1}}Y_{\ell m}(x)Y_{\ell
^{\prime }m^{\prime }}(y)Y_{\ell ^{\prime \prime }m^{\prime \prime
}}(x)Y_{\ell ^{\prime \prime }m^{\prime \prime }}(y)dxdy \\
&=&\sum_{\ell ,\ell ^{\prime },m,m^{\prime }}\sum_{q=Q+1}^{\infty
}\sum_{\ell ^{\prime \prime }m^{\prime \prime }}C_{\ell ^{\prime \prime
}}(q)\delta _{\ell }^{\ell ^{\prime \prime }}\delta _{\ell ^{\prime }}^{\ell
^{\prime \prime }}\delta _{m}^{m^{\prime \prime }}\delta _{m^{\prime
}}^{m^{\prime \prime }} \\
&=&\sum_{\ell }\sum_{q=Q+1}^{\infty }C_{\ell }(q)n_{\ell
;d} \\
&=&\sum_{q=Q+1}^{\infty }J_{q}^{2}\text{ },
\end{eqnarray*}%
where $n_{\ell ;d}$ is the dimension of the $\ell $-th eigenspace in
dimension $d$ and $\left\{ C_{\ell }(q)\right\} $ is the angular power
spectrum of $F_{q},$ see again \cite{MaPeCUP} for more discussion and
details (the discussion in this reference is restricted to $d=2,$ but the
results can be extended to any dimension).

We are left to bound $M(F_{\leq Q})$ and $C(F_{\leq Q})$.
In Section \ref{sec:M(F)} we will provide a bound for $M(F_{\leq Q})$.
Under the conditon $ Q \leq \log_3 \sqrt{n} $, such bound reduces to
\begin{equation*}
M(F_{\leq Q})
\lesssim \frac{\| \sigma \|^2}{\sqrt{n}} \sum_q^{Q} J_q^2 q 3^q \ .
\end{equation*}%
On the other hand, in Section \ref{sec:C(F)} we will show that
$$
C(F_{\leq Q}) \le M(F_{\leq Q}) \ .
$$
This completes the proof.

\subsection{Bounding $M(F_{\leq Q})$} \label{sec:M(F)}

The following proposition provides a bound on $M(F_{\leq Q})$.
The proof relies on several technical lemmas, which are given below.

\begin{proposition} \label{prop:M(F<=Q)}
We have
\begin{equation*}
M(F_{\leq Q})
\lesssim \frac{1}{\sqrt{n}} \sum_{q=0}^Q J_q^2 q 3^q \left( \| \sigma \|^2 + \frac{1}{\sqrt{n}} \sum_{q=0}^Q J_q^2 3^q \right) \ .
\end{equation*}
\end{proposition}
\begin{proof}
We have
\begin{align*}
M(F_{\leq Q}) &= \frac{1}{\sqrt{3}}\sum_{p,q}^{Q}c_{p,q}\sqrt{\mathbb{E}%
\left\Vert F_{p}\right\Vert ^{4}(\mathbb{E}\left\Vert F_{q}\right\Vert ^{4}-%
\mathbb{E}\left\Vert Z_{q}\right\Vert ^{4})} \\
&\le \sum_p^{Q} \sqrt{\mathbb{E}\left\Vert F_{p}\right\Vert ^{4}}
\sum_q^Q q \sqrt{ \mathbb{E}\left\Vert F_{q}\right\Vert ^{4}-%
\mathbb{E}\left\Vert Z_{q}\right\Vert ^{4} } .
\end{align*}
In Lemma \ref{M(F)} we compute
\begin{align*}
 \mathbb{E}\left\Vert F_{q}\right\Vert ^{4}-
\mathbb{E}\left\Vert Z_{q}\right\Vert ^{4}
= \frac{1}{n}\frac{J_{q}^{4}}{(q!)^{2}} \sum_{q_{1}=0}^{q-1}\Upsilon _{q_{1},q}\int_{\mathbb{S}^{d-1}\times
\mathbb{S}^{d-1}}\left\langle x_{1},x_{2}\right\rangle
^{2(q-q_{1})}dx_{1}dx_{2} \ ,
\end{align*}
with
\begin{align*}
\Upsilon _{q_{1},q} = \binom{q}{q_{1}}^{4}(q_{1}!)^{2}(2q-2q_{1})! \ .
\end{align*}
By Lemma \ref{BoundCountingLemma} we get the bound
\begin{equation*}
\max_{0\leq q_{1}\leq q-1}\Upsilon _{q_{1},q}\lesssim \frac{(q!)^{2} 3^{2q}}{q} \ ,
\end{equation*}
whereas Lemma \ref{BetaLemma} yields
\begin{align*}
\int_{\mathbb{S}^{d-1}\times
\mathbb{S}^{d-1}}\left\langle x_{1},x_{2}\right\rangle
^{2(q-q_{1})}dx_{1}dx_{2}
\leq 1 .
\end{align*}
Therefore,
$$
\mathbb{E}\left\Vert F_{q}\right\Vert ^{4}-
\mathbb{E}\left\Vert Z_{q}\right\Vert ^{4}
\lesssim \frac{J_{q}^{4} 3^{2q}}{n} \ .
$$
Moreover, in view of Lemma \ref{lem:EFp4}, we have
\begin{equation*}
\mathbb{E}\left\Vert F_{p}\right\Vert ^{4}
\lesssim \frac{J_{p}^{4} 3^{2p}}{n} + J_{p}^{4}
\text{\ .}
\end{equation*}%
Collecting all the terms, we finally obtain the claim
\end{proof}

In the following, we collect the technical lemmas used in the proof of Proposition \ref{prop:M(F<=Q)}.

\begin{lemma}
\label{M(F)} We have
\begin{align*}
 \mathbb{E}\left\Vert F_{q}\right\Vert ^{4}-
\mathbb{E}\left\Vert Z_{q}\right\Vert ^{4}
= \frac{1}{n}\frac{J_{q}^{4}}{(q!)^{2}} \sum_{q_{1}=0}^{q-1}\Upsilon _{q_{1},q}\int_{\mathbb{S}^{d-1}\times
\mathbb{S}^{d-1}}\left\langle x_{1},x_{2}\right\rangle
^{2(q-q_{1})}dx_{1}dx_{2}
\end{align*}
with $\Upsilon _{q_{1},q} = \binom{q}{q_{1}}^{4}(q_{1}!)^{2}(2q-2q_{1})! $.
\end{lemma}

\begin{proof}
We will write $\operatorname{Cum}(\cdot,\cdot,\cdot,\cdot)$ for the joint cumulant of four
random variables, that is,
\begin{equation*}
\operatorname{Cum}(X,Y,Z,W)=\mathbb{E}\left[ XYZW\right] -\mathbb{E}\left[ XY\right]
\mathbb{E}\left[ WZ\right] -\mathbb{E}\left[ XZ\right] \mathbb{E}\left[ WY%
\right] -\mathbb{E}\left[ XW\right] \mathbb{E}\left[ ZY\right] \text{ .}
\end{equation*}%

We have
\begin{align*}
\mathbb{E}\left\Vert F_{q}\right\Vert ^{4}
&= \frac{1}{n^{2}}\frac{J_{q}^{4}}{(q!)^{2}}%
\sum_{j_{1},j_{2},j_{3},j_{4}=1}^{n}\int_{\mathbb{S}^{d-1}\times \mathbb{S}%
^{d-1}} \\
& \mathbb{E} \left\{
V_{j_{1}}H_{q}(W_{j_{1}}x_{1}) V_{j_{2}}H_{q}(W_{j_{2}}x_{1})V_{j_{3}}H_{q}(W_{j_{3}}x_{2})V_{j_{4}}H_{q}(W_{j_{4}}x_{2}) \right\} dx_{1}dx_{2} \\
&= \frac{1}{n}\frac{J_{q}^{4}}{(q!)^{2}}\int_{\mathbb{S}^{d-1}\times \mathbb{S}%
^{d-1}} \\
& \operatorname{Cum}\left\{
V_{j}H_{q}(W_{j}x_{1}),V_{j}H_{q}(W_{j}x_{1}),V_{j}H_{q}(W_{j}x_{2}),V_{j}H_{q}(W_{j}x_{2})\right\} dx_{1}dx_{2} \\
& + \frac{J_{q}^{4}}{(q!)^{2}}\left\{ \int_{\mathbb{S}^{d-1}}\mathbb{E}\left\{
V_{j}H_{q}(W_{j}x_{1})V_{j}H_{q}(W_{j}x_{1})\right\} dx_{1}\right\} ^{2} \\
& + 2\frac{J_{q}^{4}}{(q!)^{2}}\int_{\mathbb{S}^{d-1}\times \mathbb{S}%
^{d-1}}\left\{ \mathbb{E}\left\{ V_{j}H_{q}(W_{j}x_{1})V_{j}H_{q}(W_{j}x_{2})\right\}
\right\} ^{2}dx_{1}dx_{2}\text{ .}
\end{align*}%
Now note that, in view of the normalization we adopted for the volume of $%
\mathbb{S}^{d-1}$,%
\begin{align*}
& \frac{J_{q}^{4}}{(q!)^{2}}\left\{ \int_{\mathbb{S}^{d-1}}\mathbb{E}\left\{
V_{j}H_{q}(W_{j}x_{1})V_{j}H_{q}(W_{j}x_{1})\right\} dx_{1}\right\}
^{2}=J_{q}^{4}\text{ ,} \\
& 2\frac{J_{q}^{4}}{(q!)^{2}}\int_{\mathbb{S}^{d-1}\times \mathbb{S}%
^{d-1}}\left\{ \mathbb{E}\left\{ V_{j}H_{q}(W_{j}x_{1})V_{j}H_{q}(W_{j}x_{2})\right\}
\right\} ^{2}dx_{1}dx_{2} \\
= \ & 2J_{q}^{4}\int_{\mathbb{S}^{d-1}\times \mathbb{S}^{d-1}}\left\langle
x_{1},x_{2}\right\rangle ^{2q}dx_{1}dx_{2}\text{ .}
\end{align*}%
Moreover,%
\begin{eqnarray*}
&&\int_{\mathbb{S}^{d-1}\times \mathbb{S}^{d-1}}\mathbb{E}\left\{
Z_{q}^{2}(x_{1})Z_{q}^{2}(x_{2})\right\} dx_{1}dx_{2} \\
&=&\int_{\mathbb{S}^{d-1}\times \mathbb{S}^{d-1}}\mathbb{E}\left\{
Z_{q}^{2}(x_{1})\right\} \mathbb{E}\left\{ Z_{q}^{2}(x_{2})\right\} dx_{1}dx_{2} \\
&&+2\int_{\mathbb{S}^{d-1}\times \mathbb{S}^{d-1}}\mathbb{E}\left\{
Z_{q}(x_{1})Z_{q}(x_{2})\right\} \mathbb{E}\left\{ Z_{q}(x_{1})Z_{q}(x_{2})\right\}
dx_{1}dx_{2} \\
&=& J_{q}^{4}+2J_{q}^{4}\int_{\mathbb{S}^{d-1}\times \mathbb{S}%
^{d-1}}\left\langle x_{1},x_{2}\right\rangle ^{2q}dx_{1}dx_{2}\text{ .}
\end{eqnarray*}%
Hence,%
\begin{align*}
\mathbb{E}\left\Vert F_{q}\right\Vert ^{4}-\mathbb{E}\left\Vert
Z_{q}\right\Vert ^{4} = & \frac{1}{n}\frac{J_{q}^{4}}{(q!)^{2}} \int_{\mathbb{S}^{d-1} \times \mathbb{S}%
^{d-1}} \\
& \operatorname{Cum}\left\{
V_{1}H_{q}(W_{1}x_{1}),V_{1}H_{q}(W_{1}x_{1}),V_{1}H_{q}(W_{1}x_{2}),V_{1}H_{q}(W_{1}x_{2})\right\} dx_{1}dx_{2}%
\text{ .}
\end{align*}%
Using the diagram formula for Hermite polynomials \cite[Proposition 4.15]{MaPeCUP}
and then isotropy, for $%
q_{1}+q_{2}+q_{3}+q_{4}=2q$ we have
\begin{eqnarray*}
&& \int_{\mathbb{S}^{d-1}\times \mathbb{S}^{d-1}}\operatorname{Cum}\left\{
V_{1}H_{q}(W_{1}x_{1}),V_{1}H_{q}(W_{1}x_{1}),V_{1}H_{q}(W_{1}x_{2}),V_{1}H_{q}(W_{1}x_{2})\right\} dx_{1}dx_{2} \\
&=&\sum_{q_{1}+q_{2}+q_{3}+q_{4}=2q}\Upsilon _{q_{1}q_{2}q_{3}q_{4}}\int_{%
\mathbb{S}^{d-1}\times \mathbb{S}^{d-1}}\left\langle
x_{1},x_{1}\right\rangle ^{q_{1}}\left\langle x_{1},x_{2}\right\rangle
^{q_{2}}\left\langle x_{2},x_{2}\right\rangle ^{q_{3}}\left\langle
x_{2},x_{1}\right\rangle ^{q_{4}}dx_{1}dx_{2} \\
&=&\sum_{q_{1}=0}^{q-1}\Upsilon _{q_{1},q}\int_{\mathbb{S}^{d-1}\times
\mathbb{S}^{d-1}}\left\langle x_{1},x_{2}\right\rangle
^{2(q-q_{1})}dx_{1}dx_{2}\text{ ,}
\end{eqnarray*}%
where $\Upsilon _{q_{1}q_{2}q_{3}q_{4}},\Upsilon _{q_{1},q}$ count the
possible configurations of the diagrams.
Precisely, $\Upsilon _{q_{1},q}$ is the number of connected diagrams with no flat edges
between four rows of $q$ nodes each and $q_{1}<q$ connections between
first and second row.
To compute this number explicitly,
let us label the nodes of the diagram as%
\begin{equation*}
\begin{array}{ccccccccc}
x_{1} & x_{1} & x_{1} & ... &  &  &  &  & x_{1} \\
x_{1}^{\prime } & x_{1}^{\prime } & x_{1}^{\prime } & ... &  &  &  &  &
x_{1}^{\prime } \\
x_{2} & x_{2} & x_{2} & ... &  &  &  &  & x_{2} \\
x_{2}^{\prime } & x_{2}^{\prime } & x_{2}^{\prime } & ... &  &  &  &  &
x_{2}^{\prime }%
\end{array}%
\text{ .}
\end{equation*}%
Because there cannot be flat edges, the number of edges between $x_{1}$ and $%
x_{1}^{\prime }$ is the same as the number of edges between $x_{2}$ and $%
x_{2}^{\prime }.$ Indeed, assume that the former was larger than the latter;
then there would be less edges starting from the pair $(x_{1},x_{1}^{\prime
})$ and reaching the pair $(x_{2},x_{2}^{\prime })$ than the other way
round, which is obviously absurd. There are $\binom{q}{q_{1}}$ ways to
choose the nodes of the first row connected with the second, $\binom{q}{q_{1}%
}$ ways to choose the nodes of the second connected with the first, $\binom{q%
}{q_{1}}$ ways to choose the nodes of the third connected with the fourth,
and $\binom{q}{q_{1}}$ ways to choose the nodes of the fourth connected with
the third, which gives a term of cardinality $\binom{q}{q_{1}}^{4};$ the
number of ways way to match the nodes between first and second row or third
and fourth is $(q_{1}!)^{2}.$ There are now $(2q-2q_{1})$ nodes left in the
first two rows, which can be matched in any arbitrary way with the $%
(2q-2q_{1})$ remaining nodes of the third and the fourth row; the result follows
immediately.
\end{proof}

\begin{lemma}
\label{BoundCountingLemma} The following bound holds true:
\begin{equation*}
\max_{0\leq q_{1}\leq q-1}\Upsilon _{q_{1},q}\lesssim \frac{(q!)^{2} 3^{2q}}{q}
\text{ .}
\end{equation*}
\end{lemma}

\begin{proof}
We can write%
\begin{eqnarray*}
\frac{1}{(q!)^{2}} \Upsilon _{q_{1},q} &=& \frac{1}{(q!)^{2}}\binom{q}{q_{1}}^{4}(q_{1}!)^{2}(2q-2q_{1})! \\
&=&\frac{(q!)^{2}(2q-2q_{1})!}{(q_{1}!)^{2}((q-q_{1})!)^{2}((q-q_{1})!)^{2}}=%
\binom{q}{q_{1}}^{2}\binom{2q-2q_{1}}{q-q_{1}}\text{ .}
\end{eqnarray*}%
Note that both the elements in the last expression are decreasing in $q_{1},$
when $q_{1}>\frac{q}{2}$ (say)$.$ Fix $\alpha \in \lbrack \varepsilon ,\frac{%
1}{2}+\varepsilon ],$ $\varepsilon >0;$ repeated use of Stirling's
approximation gives%
\begin{eqnarray*}
&&\frac{(q!)^{2}(2q-2q_{1})!}{(q_{1}!)^{2}((q-q_{1})!)^{2}((q-q_{1})!)^{2}}
\\
&\sim &\frac{1}{(2\pi )^{3/2}}\frac{q^{2q+1}(2q-2q_{1})^{2q-2q_{1}+\frac{1}{2%
}}e^{2q_{1}}e^{2(q-q_{1})}e^{2q-2q_{1}}}{%
e^{2q}e^{2q-2q_{1}}q_{1}^{2q_{1}+1}(q-q_{1})^{4q-4q_{1}+2}} \\
&\sim &\frac{2^{2q-2q_{1}+\frac{1}{2}}}{(2\pi )^{3/2}}\frac{q^{2q+1}}{%
q_{1}^{2q_{1}+1}(q-q_{1})^{2q-2q_{1}+\frac{3}{2}}}
\end{eqnarray*}%
Taking $q_{1}=\alpha q$ we obtain%
\begin{eqnarray*}
&&\frac{2^{2(1-\alpha )q+\frac{1}{2}}}{(2\pi )^{3/2}}\frac{q^{2q+1}}{(\alpha
q)^{2\alpha q+1}((1-\alpha )q)^{2(1-\alpha )q+\frac{3}{2}}} \\
&=&\frac{2^{2(1-\alpha )q+\frac{1}{2}}}{(2\pi )^{3/2}}\frac{1}{(\alpha
)^{2\alpha q+1}((1-\alpha ))^{2(1-\alpha )q+\frac{3}{2}}q^{\frac{3}{2}}} \\
&=&\frac{2^{\frac{1}{2}}}{(2\pi )^{3/2}q^{\frac{3}{2}}}\left(\frac{2^{1-\alpha }}{%
\alpha ^{\alpha +\frac{1}{2q}}(1-\alpha )^{1-\alpha +\frac{3}{4q}}}\right)^{2q}.
\end{eqnarray*}%
It can be immediately checked that the function $f(\alpha ):=\frac{%
2^{1-\alpha }}{\alpha ^{\alpha }(1-\alpha )^{1-\alpha }}$ admits a unique
maximum at $\alpha =\frac{1}{3},$ for which the quantity gets bounded by
$q^{-\frac{3}{2}}3^{2q}$ up to constants.
On the other hand, for $q_{1}<\left\lfloor \varepsilon q\right\rfloor $ it
suffices to notice that%
\begin{eqnarray*}
\binom{q}{q_{1}}^{2}\binom{2q-2q_{1}}{q-q_{1}} &\leq &2^{2q}\binom{q}{q_{1}}%
^{2}\leq 2^{2q}\binom{q}{\left\lfloor \varepsilon q\right\rfloor }^{2} \\
&\leq &2^{2q}\frac{q^{2q+1}}{2\pi (\varepsilon q)^{2\varepsilon
q+1}((1-\varepsilon )q)^{2(1-\varepsilon )q+1}}\text{ ,}
\end{eqnarray*}%
where we used the fact that $g(\varepsilon )=\varepsilon ^{-\varepsilon
}(1-\varepsilon )^{-(1-\varepsilon )}$ is strictly increasing in $(0,\frac{1%
}{2})$; hence we get%
\begin{equation*}
2^{2q}\frac{q^{2q+1}}{2\pi (\varepsilon q)^{2\varepsilon q+1}((1-\varepsilon
)q)^{2(1-\varepsilon )q+1}}=\frac{2^{2q}}{2\pi q}\frac{1}{((\varepsilon
)^{\varepsilon +\frac{1}{2q}}((1-\varepsilon ))^{(1-\varepsilon )+\frac{1}{2q%
}})^{2q}}\text{ .}
\end{equation*}%
The result is proved by choosing $\varepsilon $ such that
\begin{equation*}
((\varepsilon )^{\varepsilon +\frac{1}{2q}}((1-\varepsilon
))^{(1-\varepsilon )+\frac{1}{2q}})^{-1}<\frac{3}{2}\text{ .} \qedhere
\end{equation*}
\end{proof}

We recall the standard definition of the Beta
function $B(\alpha ,\beta )$:
\begin{equation*}
B(\alpha ,\beta )=\frac{\Gamma (\alpha )\Gamma (\beta )}{\Gamma (\alpha
+\beta )} \ , \quad \Gamma (\alpha )=\int_{0}^{\infty }t^{\alpha -1}\exp(-t)dt \ , \quad \alpha ,\beta >0 \ .
\end{equation*}

\begin{lemma}
\label{BetaLemma} We have%
\begin{equation*}
\int_{\mathbb{S}^{d-1}\times \mathbb{S}^{d-1}}\left\langle
x_{1},x_{2}\right\rangle ^{2(q-q_{1})}dx_{1}dx_{2}=\frac{s_{d-1}}{s_{d}}%
B\left(q-q_{1}+\frac{1}{2},\frac{d}{2}-\frac{1}{2}\right) \leq 1 \text{ .}
\end{equation*}
\end{lemma}

\begin{proof}
Fixing a pole and switching to spherical coordinates, we get%
\begin{eqnarray*}
&&\int_{\mathbb{S}^{d-1}\times \mathbb{S}^{d-1}}\left\langle
x_{1},x_{2}\right\rangle ^{2q-2q_{1}}dx_{1}dx_{2} \\
&=&\frac{s_{d-1}}{s_{d}}\int_{0}^{\pi }(\cos \theta )^{2q-2q_{1}}(\sin
\theta )^{d-2}d\theta \\
&=&\frac{s_{d-1}}{s_{d}}\int_{0}^{\pi /2}(\cos ^{2}\theta )^{q-q_{1}-\frac{1%
}{2}}(1-\cos ^{2}\theta )^{\frac{d-3}{2}}d\cos \theta \\
&=&\frac{s_{d-1}}{s_{d}}\int_{0}^{1}t^{q-q_{1}-\frac{1}{2}}(1-t)^{\frac{d-3}{%
2}}dt=\frac{s_{d-1}}{s_{d}}B\left(q-q_{1}+\frac{1}{2},\frac{d-1}{2}\right)\text{ ,}
\end{eqnarray*}
which is smaller than $1$ for all $d,q$.
\end{proof}

\begin{remark}
The bound we obtain is actually uniform over $d.$ It is likely that it could
be further improved for growing numbers of $d,$ because the Beta function
decreases quickly as $d$ diverges.
\end{remark}

\begin{lemma} \label{lem:EFp4}
We have
\begin{equation*}
\mathbb{E}\left\Vert F_{p}\right\Vert ^{4}\leq \mathbb{E}\left\Vert
F_{p}\right\Vert ^{4}-\mathbb{E}\left\Vert Z_{p}\right\Vert ^{4}+3J_{p}^{4}%
\mathbb{\ }\text{\ .}
\end{equation*}%
\end{lemma}
\begin{proof}
It suffices to observe that, following the calculations of Lemma \ref{M(F)},
\begin{equation*}
\mathbb{E}\left\Vert Z_{q}\right\Vert ^{4}=J_{q}^{4}+2J_{q}^{4}\int_{\mathbb{%
S}^{d-1}\times \mathbb{S}^{d-1}}\left\langle x_{1},x_{2}\right\rangle
^{2q}dx_{1}dx_{2}\leq 3J_{q}^{4} \text{ .} \qedhere
\end{equation*}
\end{proof}

\subsection{Bounding $C(F_{\leq Q})$} \label{sec:C(F)}


The following results reduces the problem of bounding $C(F_{\leq Q})$ to that of bounding $M(F_{\leq Q})$.

\begin{proposition}
\label{C(F)} We have
\begin{equation*}
C(F_{\leq Q})\leq M(F_{\leq Q})\text{ .}
\end{equation*}
\end{proposition}

\begin{proof}
We shall show that%
\begin{eqnarray*}
C(F_{\leq Q}) &=&\sum_{p,q:p\neq q}^{Q}c_{p,q}\sum_{p_{1}=p-q}^{p-1}\binom{p%
}{p_{1}}^{2}(p_{1}!)^{2}\binom{q}{q-p+p_{1}}%
^{2}((q-p+p_{1})!)^{2}(2(p-p_{1}))! \\
&&\times \int_{\mathbb{S}^{d-1}\times \mathbb{S}^{d-1}}\left\langle
x_{1},x_{2}\right\rangle ^{p-p_{1}}dx_{1}dx_{2} \\
&\leq& \sum_{p,q}^{Q}\sum_{p_{1}=1}^{p}c_{p,q}\binom{p}{p_{1}}%
^{4}(p_{1}!)^{2}(2p-2p_{1})!\int_{S^{d-1}\times S^{d-1}}\left\langle
x_{1},x_{2}\right\rangle ^{p-p_{1}}dx_{1}dx_{2} \\
&=& M(F_{\leq Q})\text{ .}
\end{eqnarray*}%
Recall that%
\begin{equation*}
\operatorname{Cov}(\left\Vert F_{p}\right\Vert ^{2},\left\Vert F_{q}\right\Vert ^{2})
\end{equation*}%
\begin{equation*}
=\frac{J_{p}^{2}J_{q}^{2}}{n}\int_{\mathbb{S}^{d-1}}\int_{\mathbb{S}%
^{d-1}}\operatorname{Cum}\left\{
VH_{p}(Wx_{1}),VH_{p}(Wx_{1}),VH_{q}(Wx_{2}),VH_{q}(Wx_{2})\right\}
dx_{1}dx_{2}\text{ .}
\end{equation*}%
Indeed,%
\begin{equation*}
\left\Vert F_{p}\right\Vert ^{2}=\frac{J_{p}^{2}}{n}\sum_{j_{1},j_{2}=1}^{n}%
\int_{S^{d-1}}\left\{ V_{j_{1}}H_{p}(W_{j_{1}}x)\right\} \left\{
V_{j_{2}}H_{p}(W_{j_{2}}x)\right\} dx \ ,
\end{equation*}%
and%
\begin{equation*}
\operatorname{Cov}(\left\Vert F_{p}\right\Vert ^{2},\left\Vert F_{q}\right\Vert ^{2})=%
\mathbb{E}(\left\Vert F_{p}\right\Vert ^{2}\left\Vert F_{q}\right\Vert ^{2})-%
\mathbb{E}(\left\Vert F_{p}\right\Vert ^{2})\mathbb{E}(\left\Vert
F_{q}\right\Vert ^{2}) \ ,
\end{equation*}%
where%
\begin{align*}
&\mathbb{E}(\left\Vert F_{p}\right\Vert ^{2}\left\Vert F_{q}\right\Vert ^{2})
=\frac{J_{p}^{2}J_{q}^{2}}{n^{2}}\sum_{j_{1},j_{2}=1}^{n}
\sum_{j_{3},j_{4}=1}^{n}\int_{\mathbb{S}^{d-1}}\int_{\mathbb{S}^{d-1}} \\
&\mathbb{E}\left[
V_{j_{1}}H_{p}(W_{j_{1}}x_{1})V_{j_{2}}H_{p}(W_{j_{2}}x_{1})V_{j_{3}}H_{q}(W_{j_{3}}x_{2})V_{j_{4}}H_{q}(W_{j_{4}}x_{2})%
\right] dx_{1}dx_{2} \\
=\ &\frac{J_{p}^{2}J_{q}^{2}}{n}\int_{\mathbb{S}^{d-1}\times \mathbb{S}%
^{d-1}}\operatorname{Cum}\left[
V_{1}H_{p}(W_{1}x_{1}),V_{1}H_{p}(W_{1}x_{1}),V_{1}H_{q}(W_{1}x_{2}),V_{1}H_{q}(W_{1}x_{2})%
\right] dx_{1}dx_{2} \\
+\ &J_{p}^{2}J_{q}^{2}\int_{\mathbb{S}^{d-1}\times \mathbb{S}^{d-1}}\mathbb{E}%
\left[ V_{1}H_{p}(W_{1}x_{1})V_{1}H_{p}(W_{1}x_{1})\right] \mathbb{E}\left[
V_{1}H_{q}(W_{1}x_{2})V_{1}H_{q}(W_{1}x_{2})\right] dx_{1}dx_{2} \\
+\ &2J_{p}^{2}J_{q}^{2}\int_{\mathbb{S}^{d-1}\times \mathbb{S}^{d-1}}\mathbb{E%
}\left[ V_{1}H_{p}(W_{1}x_{1})V_{1}H_{q}(W_{1}x_{2})\right] \mathbb{E}\left[
V_{1}H_{p}(W_{1}x_{1})V_{1}H_{q}(W_{1}x_{2})\right] dx_{1}dx_{2} \ .
\end{align*}
By the orthogonality of the Hermite polynomials,
the third term vanishes and we are left with%
\begin{align*}
&\frac{J_{p}^{2}J_{q}^{2}}{n}\int_{\mathbb{S}^{d-1}\times \mathbb{S}%
^{d-1}}\operatorname{Cum}\left[
V_{1}H_{p}(W_{1}x_{1}),V_{1}H_{p}(W_{1}x_{1}),V_{1}H_{q}(W_{1}x_{2}),V_{1}H_{q}(W_{1}x_{2})%
\right] dx_{1}dx_{2} \\
+\ &J_{p}^{2}J_{q}^{2}\int_{\mathbb{S}^{d-1}\times \mathbb{S}^{d-1}}\mathbb{E}%
\left[ V_{1}H_{p}(W_{1}x_{1})V_{1}H_{p}(W_{1}x_{1})\right] \mathbb{E}\left[
V_{1}H_{q}(W_{1}x_{2})V_{1}H_{q}(W_{1}x_{2})\right] dx_{1}dx_{2} \\
=\ &\frac{J_{p}^{2}J_{q}^{2}}{n}\int_{\mathbb{S}^{d-1}\times \mathbb{S}%
^{d-1}}\operatorname{Cum}\left[
V_{1}H_{p}(W_{1}x_{1}),V_{1}H_{p}(W_{1}x_{1}),V_{1}H_{q}(W_{1}x_{2}),V_{1}H_{q}(W_{1}x_{2})%
\right] dx_{1}dx_{2} \\
+\ &\left( J_{p}^{2}\int_{\mathbb{S}^{d-1}}\mathbb{E}\left[
V_{1}H_{p}(W_{1}x_{1})V_{1}H_{p}(W_{1}x_{1})\right] dx_{1}\right)
\left( J_{q}^{2}\int_{\mathbb{S}^{d-1}}\mathbb{E}\left[
V_{1}H_{q}(W_{1}x_{2})V_{1}H_{q}(W_{1}x_{2})\right] dx_{2}\right) \\
=\ &\frac{J_{p}^{2}J_{q}^{2}}{n}\int_{\mathbb{S}^{d-1}\times \mathbb{S}%
^{d-1}}\operatorname{Cum}\left[
V_{1}H_{p}(W_{1}x_{1}),V_{1}H_{p}(W_{1}x_{1}),V_{1}H_{q}(W_{1}x_{2}),V_{1}H_{q}(W_{1}x_{2})%
\right] dx_{1}dx_{2} \\
+\ &\mathbb{E}(\left\Vert F_{p}\right\Vert ^{2})\mathbb{E}(\left\Vert
F_{q}\right\Vert ^{2})\text{ .}
\end{align*}%
Indeed,%
\begin{eqnarray*}
\mathbb{E}(\left\Vert F_{p}\right\Vert ^{2}) &=&J_{p}^{2}\mathbb{E}\left[
\sum_{j_{1},j_{2}=1}^{n}\int_{S^{d-1}}\left\{
V_{j_{1}}H_{p}(W_{j_{1}}x)\right\} \left\{
V_{j_{2}}H_{p}(W_{j_{2}}x)\right\} dx\right] \\
&=&J_{p}^{2}\mathbb{E}\left[ n\int_{\mathbb{S}^{d-1}}\left\{
V_{1}H_{p}(W_{1}x)\right\} \left\{ V_{1}H_{p}(W_{1}x)\right\} dx\right]
\text{ .}
\end{eqnarray*}%
Now note that%
\begin{align*}
& \frac{J_{p}^{2}J_{q}^{2}}{n}\int_{\mathbb{S}^{d-1}\times \mathbb{S}^{d-1}}\operatorname{Cum}%
\left[
V_{1}H_{p}(W_{1}x_{1}),V_{1}H_{p}(W_{1}x_{1}),V_{1}H_{q}(W_{1}x_{2}),V_{1}H_{q}(W_{1}x_{2})%
\right] dx_{1}dx_{2} \\
=\ &\frac{J_{p}^{2}J_{q}^{2}}{n}\sum_{p_{1}=p-q}^{p-1}\binom{p}{p_{1}}%
^{2}(p_{1}!)^{2}\binom{q}{q-p+p_{1}}^{2}((q-p+p_{1})!)^{2}(2(p-p_{1}))! \\
\times\ & \int_{\mathbb{S}^{d-1}\times \mathbb{S}^{d-1}}\left\langle
x_{1},x_{2}\right\rangle ^{p-p_{1}}dx_{1}dx_{2}\text{ .}
\end{align*}%
Moreover,
\begin{equation*}
\binom{q}{q-p+p_{1}}^{2}((q-p+p_{1})!)^{2}=\frac{(q!)^{2}}{((p-p_{1})!)^{2}}%
\leq \frac{(p!)^{2}}{((p-p_{1})!)^{2}}=\binom{p}{p_{1}}^{2}(p_{1}!)^{2},
\end{equation*}%
and hence%
\begin{align*}
& \sum_{p_{1}=p-q}^{p-1}\binom{p}{p_{1}}^{2}(p_{1}!)^{2}\binom{q}{q-p+p_{1}}%
^{2}((q-p+p_{1})!)^{2}(2(p-p_{1}))!\int_{\mathbb{S}^{d-1}\times \mathbb{S}%
^{d-1}}\left\langle x_{1},x_{2}\right\rangle ^{p-p_{1}}dx_{1}dx_{2} \\
\leq\ & \sum_{p_{1}=1}^{p}\binom{p}{p_{1}}^{4}(p_{1}!)^{2}(2p-2p_{1})!\int_{%
\mathbb{S}^{d-1}\times \mathbb{S}^{d-1}}\left\langle
x_{1},x_{2}\right\rangle ^{p-p_{1}}dx_{1}dx_{2} \ ,
\end{align*}%
so that our previous bound on the fourth cumulant is sufficient, up to a
factor $\frac{J_{q}^{2}}{J_{p}^{2}}\frac{q!}{p!}\ll 1.$
\end{proof}

\subsection{Proof of Theorem \protect\ref{MainTheorem2}} \label{sec:proofthm2}

The proof of Theorem \ref{MainTheorem2} takes advantage of the tighter
bounds which are obtained in \cite[Section 4]{BourguinCampese}; we refer
to this paper and Section \ref{sec:campese}, together with the monograph \cite%
{NourdinPeccati}, for more details on the notation and further discussion.

Consider the isonormal Gaussian process with underlying Hilbert space
\begin{equation*}
\mathcal{H}:=L^{2}[0,2\pi ]\otimes L^{2}[0,2\pi ]\otimes \mathbb{R}^{d}\text{
;}
\end{equation*}%
we take%
\begin{align*}
V_{j} &= I(f_{V_{j}})=I\left(\frac{\cos (\cdot )}{\sqrt{\pi }}\otimes \frac{\exp
(ij\cdot )}{\sqrt{2\pi }}\otimes z\right)\text{ for some fixed }z\text{ such that }%
\left\Vert z\right\Vert _{\mathbb{R}^{d}}=1\text{ ,} \\
W_{j}x &= I(f_{W_{j}x})=I\left(\frac{\sin (\cdot )}{\sqrt{\pi }}\otimes \frac{%
\exp (ij\cdot )}{\sqrt{2\pi }}\otimes x\right)\text{ for any }x\in \mathbb{S}^{d-1} .
\end{align*}%
It is readily seen that these are two Gaussian, zero mean, unit variance
random variables, with covariances%
\begin{equation*}
\mathbb{E}\left[ V_{j}V_{j^{\prime }}\right] =\delta _{j}^{j^{\prime }}\text{
, }\mathbb{E}\left[ V_{j}W_{j}x\right] =0\text{ , }\mathbb{E}\left[
W_{j}x_{1}W_{j^{\prime }}x_{2}\right] =\delta _{j}^{j^{\prime }}\left\langle
x_{1},x_{2}\right\rangle _{\mathbb{R}^{d}}.\text{ }
\end{equation*}%
Also, we have that%
\begin{eqnarray*}
\frac{1}{\sqrt{n}}\sum_{j=1}^{n}V_{j}\sigma (W_{j}x) &=&\frac{1}{\sqrt{n}}%
\sum_{j=1}^{n}\sum_{p=1}^{\infty }\frac{J_{p}(\sigma )}{\sqrt{p!}}%
I(f_{V_{j}})H_{p}(W_{j}x) \\
&=&\frac{1}{\sqrt{n}}\sum_{j=1}^{n}\sum_{p=1}^{\infty }\frac{J_{p}(\sigma )}{%
\sqrt{p!}}I(f_{V_{j}})I_{p}(f_{W_{j}x}^{\otimes p})\text{ ,}
\end{eqnarray*}%
where we have used the standard identity linking Hermite polynomials and
multiple stochastic integrals (i.e., Theorem 2.7.7 in \cite{NourdinPeccati}%
). To evaluate the term $I(f_{V_{j}})I_{p}(f_{W_{j}x}^{\otimes p}),$ we
recall the product formula \cite[Theorem 2.7.10]{NourdinPeccati}%
\begin{equation*}
I_{p}(f^{\otimes p})I_{q}(g^{\otimes q})=\sum_{r=0}^{p\wedge q}r\binom{p}{r}%
\binom{q}{r}I_{p+q-2r}(f\widetilde{\otimes }_{r}g)\text{ ;}
\end{equation*}%
in our case $p=1,$ $f_{V_{j}}\widetilde{\otimes }_{1}f_{W_{j}x}^{\otimes
p}=0,$ hence we obtain%
\begin{equation*}
\frac{1}{\sqrt{n}}\sum_{j=1}^{n}I(f_{V_{j}})I_{p}(f_{W_{j}x}^{\otimes
p})=I_{p+1}(\frac{1}{\sqrt{n}}\sum_{j=1}^{n}f_{V_{j}}\widetilde{\otimes }%
_{r}f_{W_{j}x}^{\otimes p})\text{ },
\end{equation*}%
where $\widetilde{\otimes }$ denotes the symmetrized tensor product. Let us
now write%
\begin{equation*}
f_{p+1;x}:=\frac{1}{\sqrt{n}}\frac{J_{p}(\sigma )}{\sqrt{p!}}%
\sum_{j=1}^{n}f_{V_{j}}\widetilde{\otimes }_{r}f_{W_{j}x}^{\otimes p}\text{ ;%
}
\end{equation*}%
it can then be readily checked that, for $K=L^{2}(\mathbb{S}^{d-1})$ and $%
r<(p_{1}+1\wedge p_{2}+1)$%
\begin{equation*}
\left\Vert f_{p_{1}+1;x_{1}}\otimes f_{p_{2}+1;x_{2}}\right\Vert _{\mathcal{H%
}^{\otimes (p_{1}+p_{2}-2r)}}^{2}=\frac{1}{n}\frac{J_{p}^{4}(\sigma )}{%
(p!)^{2}}\left\langle x_{1},x_{2}\right\rangle ^{2r}\text{ ,}
\end{equation*}%
\begin{equation*}
\left\Vert f_{p_{1}+1;x_{1}}\otimes f_{p_{2}+1;x_{2}}\right\Vert _{\mathcal{H%
}^{\otimes (p_{1}+p_{2}-2r)}\otimes K^{\otimes 2}}^{2}=\frac{1}{n}\frac{%
J_{p}^{4}(\sigma )}{(p!)^{2}}\int_{\mathbb{S}^{d-1}\times \mathbb{S}%
^{d-1}}\left\langle x_{1},x_{2}\right\rangle ^{2r}dx_{1}dx_{2}\text{ .}
\end{equation*}%
To complete the proof, it is then sufficient to exploit  \cite[Theorem 4.3]{BourguinCampese} and to follow similar steps as in the proof
of Theorem \ref{MainTheorem}.

\appendix

\section{Appendix} \label{Appendix}

\subsection{The quantitative functional central limit theorem by Bourguin
and Campese (2020)} \label{sec:campese}

In this paper, the probabilistic distance for the distance between the
random fields we consider is the so-called $d_{2}$-metric, which is given by%
\begin{equation*}
d_{2}(F,G)=\sup_{\substack{ h\in C_{b}^{2}(K)  \\ \left\Vert h\right\Vert
_{C_{b}^{2}(K)}\leq 1}}\left\vert \mathbb{E}\left[ h(F)\right] -\mathbb{E}%
\left[ h(G)\right] \right\vert \text{ ;}
\end{equation*}%
here, $C_{b}^{2}(K)$ denotes the space of continuous and bounded
applications from the Hilbert space $K$ into $\mathbb{R}$ endowed with two
bounded Frechet derivatives $h^{\prime },h^{\prime \prime }$; that is, for
each $h\in C_{b}^{2}(K)$ there exist a bounded linear operator $h^{\prime
}:K\rightarrow \mathbb{R}$ such that $\left\Vert h^{\prime }\right\Vert
_{K}\leq 1$%
\begin{equation*}
\lim_{\left\Vert v\right\Vert \rightarrow 0}\frac{|h(x+v)-h(x)-h^{\prime
}(v)|}{\left\Vert v\right\Vert }=0\text{ ,}
\end{equation*}%
and similarly for the second derivative.

We will use a simplified version of the results by Bourguin and Campese in
\cite{BourguinCampese}, which we report below.

\begin{theorem}
(A special case of Theorem 3.10 in \cite{BourguinCampese}) Let $F_{\leq
Q}\in L^{2}(\Omega ,K)$ be a Hilbert-valued random element $F_{\leq
Q}:\Omega \rightarrow K$ be a process with zero mean, covariance operator $%
S_{\leq Q}$ and such that it can be decomposed into a finite number of
Wiener chaoses:%
\begin{equation*}
F_{\leq Q}(.)=\sum_{p=0}^{Q}F_{p}(.)\text{ .}
\end{equation*}%
Then, for $Z$ a Gaussian process on the same structure with covariance
operator $S$ we have that%
\begin{equation*}
d_{2}(F_{\leq Q},Z)\leq \frac{1}{2}\sqrt{M(F_{\leq Q})+C(F_{\leq Q})}%
+\left\Vert S-S_{\leq Q}\right\Vert _{L^{2}(\Omega ,\rm{HS})}
\end{equation*}%
where%
\begin{eqnarray*}
M(F_{\leq Q}) &=&\frac{1}{\sqrt{3}}\sum_{p,q}c_{p,q}\sqrt{\mathbb{E}%
\left\Vert F_{p}\right\Vert ^{4}(\mathbb{E}\left\Vert F_{q}\right\Vert ^{4}-%
\mathbb{E}\left\Vert Z_{q}\right\Vert ^{4})} \\
C(F_{\leq Q}) &=&\sum_{\substack{ p,q  \\ p\neq q}}c_{p,q}\operatorname{Cov}(\left\Vert
F_{p}\right\Vert ^{2},\left\Vert F_{q}\right\Vert ^{2})\text{ ,}
\end{eqnarray*}%
$Z_{q}$ a centred Gaussian process with the same covariance operator as $%
F_{q}$ ($\mathbb{E}\left[ Z_{q}(x_{1})Z_{q}(x_{2})\right] =J_{q}^{2}(\sigma
)\left\langle x_{1},x_{2}\right\rangle ^{q}$) and%
\begin{equation*}
c_{p,q}=
\begin{cases}
1+\sqrt{3} & p=q \\
\frac{p+q}{2p} & p\neq q%
\end{cases}%
.
\end{equation*}
\end{theorem}

\begin{remark}
The general version of Theorem 3.10 in \cite{BourguinCampese} covers a
broader class of processes which can be expanded into the eigenfunctions of
Markov operators. We do not need this extra generality, and we refer to \cite%
{BourguinCampese} for more discussion and details.
\end{remark}

We will now review another result by \cite{BourguinCampese}, which holds
under tighter smoothness conditions. We shall omit a number of details, for
which we refer to classical references such as \cite{NourdinPeccati}.

Given a Hilbert space $\mathcal{H}$ we recall the isonormal process is the
collection of zero mean Gaussian random variables with covariance function%
\begin{equation*}
\mathbb{E}\left[ X(h_{1})X(h_{2})\right] =\left\langle
h_{1},h_{2}\right\rangle _{\mathcal{H}}.
\end{equation*}%
In our case these random variables take values in the separable Hilbert
space $L^{2}(\Omega ,\mathbb{S}^{d-1}).$ For smooth functions $F:\Omega
\rightarrow L^{2}(\Omega ,\mathbb{S}^{d-1})$ of the form
\begin{equation*}
F=f(W(h_{1}),...,W(h_{p}))\otimes v\text{ },\text{ }f\in C_{b}^{\infty }(%
\mathbb{R}^{p})\text{ , }v\in L^{2}(\Omega ,\mathbb{S}^{d-1})\text{ ,}
\end{equation*}%
we recall that the Malliavin derivative is defined as%
\begin{equation*}
DF=\sum_{i=1}^{p}\partial _{i}f(W(h_{1}),...,W(h_{p}))h_{i}\otimes v
\end{equation*}%
whose domain, denoted by $\mathbb{D}^{1,2},$ is the closure of the space of
smooth functions with respect to the Sobolev norm $\left\Vert F\right\Vert
_{L^{2}(\Omega ,\mathbb{S}^{d-1})}^{2}+\left\Vert DF\right\Vert
_{L^{2}(\Omega ,\mathcal{H\otimes }\mathbb{S}^{d-1})}^{2};$ $\mathbb{D}%
^{1,4} $ is defined analogously.

In this setting, the Wiener chaos decompositions take the form%
\begin{equation*}
F=\sum_{p=1}^{\infty }I_{p}(f_{p})\text{ , }f_{p}\in \mathcal{H}^{\odot
p}\otimes L^{2}(\mathbb{S}^{d-1})\text{ },
\end{equation*}%
where $\mathcal{H}^{\odot p}$ denotes the $p$-fold symmetrized tensor
product of $\mathcal{H},$ see \cite{BourguinCampese}, Subsection 4.1.2. The
main result we are going to exploit is their Theorem 4.3, which we can
recall as follows.

\begin{theorem}
(A special case of Theorem 4.3 in \cite{BourguinCampese}) Let $Z$ be a
centred random element of $L^{2}(\mathbb{S}^{d-1})$ with covariance operator
$S$ and $F\in \mathbb{D}^{1,4}$ with covariance operator $T$ and chaos
decomposition $F=\sum_{p}I_{p}(f_{p}),$ where $f_{p}\in \mathcal{H}^{\odot
p}\otimes L^{2}(\mathbb{S}^{d-1}).$ Then%
\begin{equation*}
d_{2}(F,Z)\leq \frac{1}{2}(\widetilde{M}(F)+\widetilde{C}(F)+\left\Vert
S-T\right\Vert _{\rm{HS}})\text{ ,}
\end{equation*}%
where%
\begin{eqnarray*}
\widetilde{M}(F) &=&\sum_{p=1}^{\infty }\sqrt{\sum_{r=1}^{p-1}\widetilde{%
\Upsilon }_{p,p}^{2}(r)\left\Vert f_{p}\otimes _{r}f_{p}\right\Vert _{%
\mathcal{H}^{\otimes (2p-2r)}\otimes L^{2}(\mathbb{S}^{d-1})^{\otimes 2}}^{2}%
}\text{ ,} \\
\widetilde{C}(F) &=&\sum_{1\leq p,q\leq \infty ,\text{ }p\neq q}^{\infty }%
\sqrt{\sum_{r=1}^{p\wedge q}\widetilde{\Upsilon }_{p,q}^{2}(r)\left\Vert
f_{p}\otimes _{r}f_{q}\right\Vert _{\mathcal{H}^{\otimes (p+q-2r)}\otimes
L^{2}(\mathbb{S}^{d-1})^{\otimes 2}}^{2}}\text{ ,}
\end{eqnarray*}%
and%
\begin{equation*}
\widetilde{\Upsilon }_{p,q}(r)=p^{2}(r-1)!\binom{p-1}{r-1}\binom{q-1}{r-1}%
(p+q-2r)!\text{ .}
\end{equation*}
\end{theorem}

\subsection{The ReLu activation function}

We consider here the most popular activation function, i.e., the standard
ReLu defined by $\sigma (t)=t\mathbb{I}_{[0,\infty )}(t)$. The Hermite
expansion is known to be given by (see for instance \cite{Eldan}, Lemma 17,
or \cite{Klukowski}, Theorem 2 and \cite{Goel,Daniely}):

\begin{equation*}
J_{q}(\sigma )=
\begin{cases}
\frac{1}{\sqrt{2\pi }} & q=0\text{ ,} \\
\frac{1}{2} & q=1\text{ ,} \\
0 & q>1\text{ , }q\text{ odd} \\
\frac{(-1)^{\frac{q}{2}+1}(q-3)!!}{\sqrt{\pi }\sqrt{q!}} & q\text{
even}%
\end{cases}%
.
\end{equation*}%
The following Theorem is standard (compare \cite{Klukowski}), but we include
it for completeness.

\begin{lemma} \label{lem:relu}
As $q\rightarrow \infty $%
\begin{equation*}
J_{q}^{2}\sim \frac{\sqrt{2}}{\sqrt{\pi ^{3}}q^{5/2}}\text{ }.
\end{equation*}
\end{lemma}

\begin{proof}
The result follows from a straightforward application of Stirling's formula,
which gives%
\begin{eqnarray*}
q! &\sim &\sqrt{2\pi }q^{q+\frac{1}{2}}\exp (-q)\text{ ,} \\
(q-3)!! &=&\frac{(q-3)!}{2^{\frac{q}{2}-2}(\frac{q}{2}-2)!}\sim \frac{%
(q-3)^{q-\frac{5}{2}}\exp (-q+3)}{2^{\frac{q}{2}-2}(\frac{q}{2}-2)^{\frac{q}{%
2}-\frac{3}{2}}\exp (-\frac{q}{2}+2)} \\
&=&\frac{\exp (-\frac{q}{2}+1)}{(q-3)^{1/2}(\frac{q}{2}-2)^{1/2}}(1+\frac{1}{%
q-4})^{\frac{q}{2}-2}(q-3)^{\frac{q}{2}},
\end{eqnarray*}%
so that
\begin{align*}
\frac{((q-3)!!)^{2}}{\pi (q)!} &\sim \frac{\frac{\exp (-q+2)}{(q-3)(\frac{q%
}{2}-2)}(1+\frac{1}{q-4})^{q-4}(q-3)^{q}}{\sqrt{2\pi ^{3}}(q)^{q+\frac{1}{2}%
}\exp (-q)} \\
&\sim \frac{\exp (3)}{\sqrt{2\pi ^{3}}(q-3)(\frac{q}{2}-2)\sqrt{q}}(1-\frac{%
3}{q})^{q}\sim \frac{2^{1/2}}{\sqrt{\pi ^{3}}(q)^{5/2}}\text{ .} \qedhere
\end{align*}
\end{proof}

\begin{remark}
The corresponding covariance kernel is given by, for any $x_{1},x_{2}\in
\mathbb{S}^{d-1}$%
\begin{eqnarray*}
&&\mathbb{E}[\sigma (W^{T}x_{1})\sigma (W^{T}x_{2})] \\
&=&\frac{1}{2\pi }+\frac{\left\langle x_{1},x_{2}\right\rangle }{4}+\frac{%
\left\langle x_{1},x_{2}\right\rangle ^{2}}{4\pi }+\frac{1}{2\pi }%
\sum_{q=2}^{\infty }\frac{((2q-3)!!)^{2}}{(2q)!}\left\langle
x_{1},x_{2}\right\rangle ^{2q} \\
&=&\frac{1}{\pi }(u(\pi -\arccos u)+\sqrt{1-u^{2}})\text{ ,}
\end{eqnarray*}%
for $u=\left\langle x_{1},x_{2}\right\rangle ,$ see also \cite{BachBietti}.
\end{remark}

\begin{remark}
The rate for $J_{q}$ in Lemma \ref{lem:relu} is consistent with the one obtained by \cite%
{Klukowski}. In \cite{Eldan}, $J_{q}^{2}=O(q^{-3})$ is given instead,
yielding in \cite[Theorem 3]{Eldan} the rate
$$
\left( \frac{\log d\times \log \log n}{\log n}\right) .
$$
According to Lemma \ref{lem:relu}, this rate becomes
\begin{equation*}
\left( \frac{\log d\times \log \log n}{\log n}\right) ^{3/4} \ ,
\end{equation*}
which is the one we actually report in Table \ref{savare}.
\end{remark}

\end{document}